\documentclass[12pt, a4paper, reqno]{amsart}

\usepackage{amsmath}
\usepackage{amsfonts}
\usepackage{amssymb}
\usepackage{amsthm}
\usepackage[latin1]{inputenc}
\usepackage{enumerate}

\theoremstyle{plain}
\newtheorem{theo}[equation]{Theorem}

\newtheorem{cor}[equation]{Corollary}
\newtheorem{corollary}[equation]{Corollary}
\newtheorem{lem}[equation]{Lemma}
\newtheorem{lemma}[equation]{Lemma}
\newtheorem{prop}[equation]{Proposition}

\theoremstyle{definition}
\newtheorem{defi}[equation]{Definition}

\numberwithin{equation}{section}
\def\kint_#1{\mathchoice%
          {\mathop{\kern 0.2em\vrule width 0.6em height 0.69678ex depth -0.58065ex
                  \kern -0.8em \intop}\nolimits_{\kern -0.4em#1}}%
          {\mathop{\kern 0.1em\vrule width 0.5em height 0.69678ex depth -0.60387ex
                  \kern -0.6em \intop}\nolimits_{#1}}%
          {\mathop{\kern 0.1em\vrule width 0.5em height 0.69678ex depth -0.60387ex
                  \kern -0.6em \intop}\nolimits_{#1}}%
          {\mathop{\kern 0.1em\vrule width 0.5em height 0.69678ex depth -0.60387ex
                  \kern -0.6em \intop}\nolimits_{#1}}}
\def\vintslides_#1{\mathchoice%
          {\mathop{\kern 0.1em\vrule width 0.5em height 0.697ex depth -0.581ex
                  \kern -0.6em \intop}\nolimits_{\kern -0.4em#1}}%
          {\mathop{\kern 0.1em\vrule width 0.3em height 0.697ex depth -0.604ex
                  \kern -0.4em \intop}\nolimits_{#1}}%
          {\mathop{\kern 0.1em\vrule width 0.3em height 0.697ex depth -0.604ex
                  \kern -0.4em \intop}\nolimits_{#1}}%
          {\mathop{\kern 0.1em\vrule width 0.3em height 0.697ex depth -0.604ex
                  \kern -0.4em \intop}\nolimits_{#1}}}

\newcommand{\vp}{\varphi}
\newcommand{\eps}{\varepsilon}

\newcommand{\R}{\mathbf{R}}
\newcommand{\Rn}{\R^n}

\newcommand{\dist}{\operatorname{dist}}
\newcommand{\norm}[1]{\left|\left| #1 \right|\right|}
\newcommand{\abs}[1]{\left| #1 \right|}

\renewcommand{\limsup}{\operatornamewithlimits{lim \, sup}}
\renewcommand{\liminf}{\operatornamewithlimits{lim \, inf}}

\newcommand{\essinf}{\operatornamewithlimits{ess\,inf}}

\newcommand{\supp}{\operatorname{spt}}
\newcommand{\spt}{\operatorname{spt}}

\newcommand{\parts}[2]{\frac{\partial {#1}}{\partial {#2}}}

\newcommand{\theoref}[1]{Theorem~\ref{#1}}

\newcommand{\lemref}[1]{Lemma~\ref{#1}}

\newcommand{\defiref}[1]{Definition~\ref{#1}}

\newcommand{\corref}[1]{Corollary~\ref{#1}}

\newcommand{\ud}{\,\mathrm{d}}

\newcommand{\trm}{\textrm}

\newcommand{\half}{{\frac{1}{2}}}

\newcommand{\ol}{\overline}

\def\XXint#1#2#3{{\setbox0=\hbox{$#1{#2#3}{\int}$}
     \vcenter{\hbox{$#2#3$}}\kern-.5\wd0}}

\newcommand{\Om}{\Omega}

\newcommand{\essliminf}{\operatornamewithlimits{ess\,lim\,inf}}

\title[Irregular obstacles]{Irregular Time
  Dependent Obstacles}

\author{Peter Lindqvist and Mikko Parviainen}

\address{Peter Lindqvist,
\hfill\break\indent Department of Mathematics \hfill\break\indent
Norwegian University of Science and Technology \hfill\break\indent N-7491, 
\hfill\break\indent Trondheim, Norway
\hfill\break\indent {\tt lqvist@math.ntnu.no}}

\address{Mikko Parviainen
\hfill\break\indent Institute of Mathematics \hfill\break\indent
Aalto University \hfill\break\indent P.O. Box
111100, FIN-00076 Aalto, \hfill\break\indent Espoo, Finland
\hfill\break\indent {\tt mikko.parviainen@tkk.fi}}

\date{\the\day.\the\month.\the\year}



\begin{document}
\begin{abstract} We study the obstacle problem for the Evolutionary p-Laplace
  Equation when the obstacle is discontinuous and without
  regularity in the time variable. Two quite different procedures
  yield the same solution.
\end{abstract}

\thispagestyle{headings}\maketitle

\section{Introduction}

Our objective is the obstacle problem for the Evolutionary $p$-Laplace
Equation in the slow diffusion case $p > 2.$ The appearing functions are forced to lie almost everywhere
above a given function, the \textsf{obstacle} $\psi.$ Our emphasis is on  very
irregular obstacles. Then some uniqueness and convergence results, known in the
stationary case, are no longer valid in the
parabolic theory. Thus some precaution is called for.

The weak solutions and weak supersolutions of the Evolutionary $p$-Laplace Equation
$$\frac{\partial u}{\partial t} = \mathrm{div}(|\nabla u|^{p-2}\nabla
u)$$
are \emph{a priori} required to belong to the Sobolev space
$L^p(0,T;W^{1,p}(\Omega)).$ Therefore
it is natural to treat the obstacle problem under the assumption that
the obstacle $\psi$ belongs to the same space. Needless to say, when
it comes to the basic theory, it is
very important that no
further assumptions be imposed on the obstacle. However, the
natural    \begin{equation*}
\text{\textsf{Assumption:}}\quad\psi \in L^p(0,T;W^{1,p}(\Omega))
\end{equation*}
does not include any requirements about the time derivative
$\frac{\partial \psi}{\partial t}.$ Neither must $\psi$ be continuous.
Indeed,  for instance rather irregular
dis\-con\-tinu\-ous functions of the type $\psi(x,t) = \psi(t)$ belong to
this space.  The variational problem is difficult
to handle under this general assumption. In the literature, so far as
we know, extra conditions about the ``missing'' time derivative or
other devices to control the time behavior are always present.  In the
present work, we carefully avoid such additional regularity
assumptions, but for convenience we
 require that the obstacle $\psi$ is bounded and
of compact support.

Given a general obstacle $\psi$,  belonging to the natural space
mentioned above, we will define the solution of
the obstacle problem in two different ways:
\begin{itemize}
\item{\textsf{the least solution} $w^*.$} This comes from the
 pointwise infimum of
  weak supersolutions lying above the obstacle almost everywhere.
\medskip
\item{\textsf{the variational solution} $v$.} The obstacle $\psi$ is
approximated by  time convolutions $\psi_{\varepsilon}$ and these act as
obstacles. The limit of the solutions of the approximating obstacle
problems is the variational solution $v.$
\end{itemize}
 We prove that the least solution and the variational solution
coincide (Theorem 4.10). Since $w^*$ is unique by its definition,
it follows that also the variational solution is unique. The
uniqueness of $v$ is, as it were, difficult to achieve without evoking
$w^*$. Furthermore, the variational inequality
\begin{gather}
\int_{0}^{T}\!\!\int_{\Omega}\Bigl(|\nabla v|^{p-2}\nabla v \cdot
\nabla(\phi - v) + (\phi - v)\frac{\partial \phi}{\partial
  t}\Bigr)\,dx\,dt \nonumber \\
\geq \frac{1}{2} \int_{\Omega}|\phi(x,T)-v(x,T)|^2dx
\end{gather}
holds for all \emph{smooth} $\phi,$\quad$\phi \geq \psi$ a.e.
 and $\phi = \psi$
 on the parabolic boundary\footnote{The reader may notice that,
   strictly speaking not
   even the obstacle $\psi$ itself, is always admissible as a test
   function in (1.1).}. The same holds for $w^*,$ since $v =
w^*.$ However,
 in the presence of an irregular obstacle, the above variational
 inequality also can have "false solutions": uniqueness fails at this
 level\footnote{A counterexample is presented in section 5}.
 Therefore the procedure with the convolutions $\psi_{\eps}$ is
 decisive;
the $\psi_{\eps}$'s capture the time behavior of their limit $\psi.$

We seize the opportunity to mention the celebrated Lavrentiev
phenomenon. If the obstacle $\psi$ is not upper semicontinuous, one
cannot always reach the least solution by using merely
\emph{continuous} weak supersolutions $u$ satisfying
$u \geq \psi.$ Neither can one  in the construction of the
variational solution, restrict oneself to approximants
satisfying $\psi_{j} \geq \psi$ almost everywhere. See section 5. This excludes
some  easy definitions.

We emphasize that this is not the theory about \emph{thin obstacles},
where the functions are forced to lie above the obstacle at
\emph{each} point. Our inequalities are
usually valid only
almost everywhere and no finer theory about capacities is used. ---It
has not escaped our notice that the results suggest a generalization
to other equations of the same structural type. Also the wider range
 $p > 2n/(n+2)$  of exponents could be included.

\section{Preliminaries}

We consider the domain
\[
\Om_T=\Om\times(0,T),
\]
where $\Om$ is a regular and bounded domain in $\R^n$, for example a
ball will do.
Its parabolic boundary is
\[
\partial_p \Om_T=(\ol \Om\times\{0\})\cup (\partial \Om\times[0,T]).
\]
Let
\[
B=B_R(x_0)=\{x\in \Rn\,:\, \abs{x-x_0}<R\}
\]
denote the ball of radius $r$ centered at $x$. The space-time cylinders
\[
\begin{split}
Q =Q_r(x,t)=B_r(x)\times(t-r^p,t+r^p).
\end{split}
\]
are convenient for some limit procedures.

  As usual, $W^{1,p}(\Om)$
denotes the Sobolev space of those real-valued functions $f$  that
together with their  distributional first partial derivatives $\partial
f /\partial x_i$, $i=1,2,\dots,n$, belong to
$L^p(\Om)$. We use the norm
\[
\|f\|_{W^{1,p}(\Om)} =\left(\int_\Om(|f|^p
+|\nabla f|^p)\,dx\right)^{1/p}.
\]
The Sobolev space  $W_0^{1,p}(\Om)$  with zero boundary values is
the closure of $C_0^\infty(\Om)$ with respect to the Sobolev norm.

The Sobolev space
\[
L^p(0,T;W^{1,p}(\Om)),
\]
consists  of all functions $u(x,t)$ such that $u(x,t)$ belongs to $W^{1,p}(\Om)$ for almost every $0< t <
T$,  $u(x,t)$ is measurable as a mapping from $(0,T)$ to $W^{1,p}(\Om)$, and the norm
\[
\left(\int\!\!\int_{\Om_T}( |u(x,t)|^{p} + |\nabla
u(x,t)|^{p})\,dx\,dt\right)^{1/p}
\]
is finite. The definition of the space
$L^p(0,T;W_{0}^{1,p}(\Om))$ is analogous.

\begin{defi}
 A function $u\in L^p_\trm{loc}(0,T;W^{1,p}_\trm{loc}(\Om))$
is a \emph{weak supersolution}  to the $p$-parabolic equation, if
\begin{equation}
\label{eq:weak-super}
\int\!\!\int_{\Om_T}\left(\abs{\nabla u}^{p-2} \nabla u \cdot \nabla
\vp - u \parts{\vp}{t}\right)\!\ud x \ud t \geq 0
\end{equation}
for every $\vp \in C_{0}^{\infty}({\Om_T}),\, \vp \geq 0$. It is a
\emph{weak subsolution}, if the integral is non-positive. A function $u$ is a
\emph{weak solution} if it is both a super- and  a subsolution, that is,
\begin{equation}
\label{eq:weak-sol}
\int\!\!\int_{\Om_T}\left(\abs{\nabla u}^{p-2} \nabla u \cdot \nabla
\vp - u \parts{\vp}{t}\right)\!\ud x \ud t=0
\end{equation}
for every $\vp \in C_{0}^{\infty}({\Om_T})$.
\end{defi}

By parabolic regularity theory,
 a continuous representative of a weak solution  always exists.  It is here
 called a \emph{$p$-parabolic function}. For the theory of weak solutions
 the reader may
 consult  \cite{dibenedetto93} and \cite{wuzyl01}.

 We shall use the regularizations
\[
\begin{split}
w^*(x,t)=\essliminf_{(y,s)\to(x,t)} w(y,s)= \lim_{r \to 0}
(\essinf_{Q_r(x,t)\cap \Om_T }w)
\end{split}
\]
and
\[
\begin{split}
\hat w(x,t)=\liminf_{(y,s)\to (x,t)}w(y,s)=\lim_{r \to 0}
(\inf_{Q_r(x,t)} w).
\end{split}
\]
Both are lower semicontinuous.

The lower semicontinuity of $w^*$ follows from the definition
in a straightforward manner: Fix $(x,t)\in \Om_T$. Then for every $\eps>0$, we may choose a radius $r>0$ such that $Q_r(x,t)\subset \Om_T$ and
\[
\abs{ w^*(x,t) -\essinf_{Q_r(x,t)} w}
\leq \eps.
\]
Choose $(y,s)\in Q_r(x,t)$ and observe that for all small enough $\rho>0$, we have $Q_\rho(y,s)\subset
Q_r(x,t)$. Thus,
\[
w^*(y,s) \geq \essinf_{Q_r(x,t)} w\geq w^*(x,t)
-\eps.
\]
Since $\eps>0$ was arbitrary, this leads to
\[
\liminf_{(y,s)\to (x,t)}w^*(y,s)\geq w^*(x,t)
,
 \]
which proves the assertion. The proof at the boundary is analogous.
\normalsize

According to \cite{kuusi09} the $\essliminf$-regularization of a weak
supersolution coincides with the original function almost everywhere,
and thus \emph{every weak supersolution has a lower semicontinuous
  representative.}

Let us now introduce the \emph{obstacle} $\psi$. In this section it is
only assumed to be a measurable function satisfying the inequality
 $0\le \psi \le L$ in $\Omega_T.$
\begin{defi}
\label{def:irregular-obstacle}
Let $\psi$ be the obstacle and consider the class
\[
\begin{split}
\mathcal S_{\psi}=\{u\,:\, u&\textrm{ is $\essliminf$-regularized weak supersolution},\\
&\hspace{14 em}u\geq\psi \,\, \textrm{a.e.\ in}\ \Omega_{T}\}.
\end{split}
\]
Define the function
\[
w(x,t)=\inf_{u} u(x,t),
\]
where the infimum is taken over the whole class $\mathcal
S_{\psi}$. We say that its regularization  $w^*(x,t)$ is  the \emph{least solution} to the obstacle
problem\footnote{In Potential Theory, $w^*$ is often called the \emph{balayage}.}.
\end{defi}

  The least solution always exists and is unique.
If $u_1, u_2 \in \mathcal
S_{\psi}$, then also their pointwise minimum
$\min\{u_1, u_2\}$ belongs to $\mathcal S_{\psi}$, cf.\ for example
 Lemma~3.2.\ in \cite{kortekp10}.
 Therefore Choquet's
well known topological lemma is applicable.
\begin{lem}[Choquet]
\label{lem:choquet}
Let $w$ be as above. There exists a decreasing sequence of functions in $\mathcal{S}_\psi$ converging pointwise to a function $u$ such that
\[
\begin{split}
\hat  u(x,t)=\hat w(x,t)
\end{split}
\]
at every point in $\Om_T$.
\end{lem}
Next we recall Theorem~4.3 from \cite{kinnunenlp10}, based on
Theorem~6 in  \cite{lindqvistm07}, \cite{simon87}, and Theorem~5.3.\ in \cite{kortekp10}.  An essential
ingredient in the proof is  that  a Radon measure is assigned to every weak supersolution.

\begin{theo}\label{thm:superparab-compactness}
  Let $u_i$ be a bounded sequence of weak supersolutions in $\Om_T$.  Then there exist a weak supersolution $u$ and a subsequence, still denoted by $u_i$, such that
  \[
  u_i\to u , \quad
  \nabla u_i\to \nabla u \quad \textrm{a.e.\
    in}\quad \Om_T.\]
%    and
%    \[
%    \begin{split}
%    u_i\to u \quad \trm{in}\quad L^q_\trm{loc}(0,T;W^{1,q}_\trm{loc}(B))\quad \trm{for every}\quad 1\leq q<p.
%    \end{split}
%    \]
\end{theo}

In \lemref{thm:obstacle-sol-supersol}, we  will show that the least
solution $w^*$ to the obstacle problem is a weak supersolution. The proof
is based on Choquet's   lemma and the above convergence result. Since
Choquet's  lemma is formulated for $\liminf$-regularizations, while the
 definition of a least solution uses the $\essliminf$-regularization,
 we  show that for the infimum $w$ these coincide.
\begin{lem}
\label{lem:regularizations-coincide}
For the least solution it holds everywhere that
\[
\begin{split}
w^*=\hat w.
\end{split}
\]
\end{lem}
\begin{proof}
Clearly $\hat w\leq w^*$, and it remains to show that $w^*\leq \hat
w$. First, notice that $w^* \leq w$. Indeed,
$$w^* = \essliminf w \leq \essliminf u = u$$
for each admissible $\essliminf$-regularized $u$, hence $w^* \leq
\inf \{u\} = w.$ Using this and the semicontinuity of  $w^*$, we obtain
$$\hspace{7.5 em}w^* \leq \liminf w^* \leq \liminf w = \hat w.\hspace{7.5 em}
\qedhere$$
\end{proof}

\begin{theo} \label{thm:obstacle-sol-supersol}
The least solution $w^*$ with the obstacle $\psi$ is a weak supersolution. Furthermore, $w=w^*$ almost everywhere.
\end{theo}
\begin{proof}
By \lemref{lem:choquet}, there exists a decreasing sequence in
$\mathcal S_{\psi}$ converging to a function $u$ so that
\[
\begin{split}
\hat  u(x,t)=\hat w(x,t)
\end{split}
\]
at each point.
By Theorem~\ref{thm:superparab-compactness} one can pass to the limit
under the integral sign in (\ref{eq:weak-super}), whence the limit
$u$ is a weak supersolution. It follows that
\[
\begin{split}
u^*=u
\end{split}
\]
almost everywhere. The proof of \lemref{lem:regularizations-coincide} also applies to $u$ and thus, $\hat u=u^*$ and $\hat w=w^*$.
Clearly, $u\geq w$. It follows that
\[
\begin{split}
\hat w=\hat u= u^*=u\geq w\geq \hat w
\end{split}
\]
almost everywhere, and since $w^*=\hat w$, this implies that $w=w^*$ almost everywhere.
\end{proof}

\section{Continuous obstacles}

In this section we  consider
\emph{continuous} obstacles.
However, we do not assume that the obstacle has a time derivative.

 We prove that if the obstacle is continuous, so is $w^*$, and that
 $w^*$ is even $p$-parabolic  in the set where  the obstacle does not
 hinder. For  the elliptic case, see
 \cite{kilpelainen89}. In the proof, we use a so-called Poisson modification.
\begin{defi}
Let  $Q\Subset \Om_T$ and let $w$ be a bounded and
$\essliminf$-regularized
 supersolution. We define its  \emph{Poisson modification} with respect to $Q$ as
\[
\begin{split}
w_P(x,t)=\begin{cases}
w,& \trm{in}\quad \Om_T\setminus  Q\\
v,& \trm{in}\quad Q,
\end{cases}
\end{split}
\]
where
\[
\begin{split}
v(\xi)=\sup\{h(\xi)\,:\,h\in C(\ol Q) \trm{ is }p\trm{-parabolic and }h\leq w\trm{ on }\partial_p Q\}.
\end{split}
\]
\end{defi}
As shown in Section~4.6.\ in \cite{kilpelainenl96}, $w_P$ is
$p$-parabolic in $Q$. Obviously, $w_P$ is lower
semicontinuous. Always,  $w_P \leq w$ by the Comparison Principle.

\begin{theo}
\label{thm:continuous}
 Let $\psi \in C(\ol \Om_T)$. The least solution $w^*$  with the obstacle $\psi$ is
continuous up to the boundary, and $w^*=\psi$ at $\partial_p \Om_T$. Moreover, $w^*$ is $p$-parabolic in the open set
$\{w^*>\psi\}$.
\end{theo}

\begin{proof}
Since $w^*=\hat w$, we can work with $\hat w$. Since $\hat w$ is lower semicontinuous, it remains to show that $\hat w$ is upper semicontinuous. To establish this, fix $(x_0,t_0)\in
\Om_T$ and observe that by the lower semicontinuity of
$\hat w$ and the continuity of $\psi$, there exists a cylinder  $Q=Q(x_0,t_0)\Subset \Om_T$ such that
\[
\hat w+\eps\geq \psi(x_0,t_0)+\frac{\eps}{2}\geq \psi\quad \trm{on}\quad \ol Q.
\]
Notice also that $\hat w+\eps$ is a supersolution. Let $w_P$ be the Poisson modification of
$\hat w$ in $Q$.
Since $w_P + \eps$ is $p$-parabolic in $Q$ and
$w_P + \eps\geq \psi(x_0,t_0)+\frac{\eps}{2}$ at $\partial_p Q$, it follows by comparison  that
\[
w_P + \eps \geq \psi(x_0,t_0)+\frac{\eps}{2}\geq \psi\quad \trm{in}\quad  Q,
\]
and hence,
\[
\begin{split}
w_P + \eps \geq\psi\quad  \trm{in}\quad \Om_T.
\end{split}
\]
Thus $w_P + \eps$ an admissible test
function in $\mathcal{S}_\psi$. This implies that
\[
\hat w\leq w_P + \eps
\]
in $\Omega_T$.
Hence
\[
\begin{split}
\limsup_{(y,s)\to (x_0,t_0)}\hat w(y,s)&\leq \lim_{(y,s)\to (x_0,t_0)}
w_P(y,s) + \eps \\
&=w_P(x_0,t_0) + \eps \leq \hat w(x_0,t_0)+\eps.
\end{split}
\]

Since $\eps>0$ was arbitrary, this shows that $\hat w$ is upper semicontinuous
at $(x_0,t_0)$ and, as it  is also lower semicontinuous, it is continuous
at the point $(x_0,t_0)$.

To see that $w^*$ is continuous up to the boundary, we use a barrier
argument as in \cite{kilpelainenl96}.  Let $(x_0,t_0)\in \partial_p \Om$. Since the boundary is regular, there
exists a closed  $n+1$-dimensional  ball
\[
\begin{split}
\{(x,t)\,:\,\abs{x-x'}^2+(t-t')^2\leq R_0^2\}
\end{split}
\]
in the complement that intersects the closure $\ol \Om_T$ exactly at $(x_0,t_0)$.
Then  the function
\[
\begin{split}
f(x,t)=e^{-\alpha R_0^2}-e^{-\alpha R^2},\quad R=\sqrt{\abs{x-x'}^2+(t-t')^2}
\end{split}
\]
with a suitable constant $\alpha>0$ is a supersolution. The function $f$ takes the value $0$ at $(x_0,t_0)$ and is positive in $\ol \Om_T\setminus\{(x_0,t_0)\}$. Then for any $\eps$ there exists $\lambda>0$ such that
\[
\begin{split}
\eps+\psi(x_0,t_0)+\lambda f(x,t)
\end{split}
\]
is a supersolution and is greater than or equal to $\psi(x,t)$ on $\ol
\Om_T$. By comparison
$$\psi(x,t) \leq w^*(x,t) \leq \eps + \psi(x_0,t_0)+\lambda f(x,t).$$
Since $\eps>0$ is arbitrary, this implies that $w^*$ is
continuous up to the boundary, and that $w^*=\psi$ on $\partial_p
\Om_T$. Observe that the calculation omitted  above is delicate: in general, supersolutions cannot be multiplied by constants.

Finally, we show that $\hat w$ is $p$-parabolic in $\{\hat w>\psi\}$. Indeed, for each $(x_0,t_0)\in \{\hat w>\psi\}$, there exists $\lambda>0$ and a
cylinder $Q=Q(x_0,t_0)\Subset\{\hat w>\psi\} $ such that
\[
\hat w>\lambda> \psi
\]
in $Q$. But now for the Poisson modification  $\hat w_P$ of $\hat w$ in $Q$, we have
\[
\hat w \geq \hat w_P>\lambda> \psi.
\]
This implies that $w_P=\hat w$ since  $ \hat w$ was the infimum, and thus $\hat w$ is
$p$-parabolic in $Q$.
\end{proof}

 Next we define a \emph{variational solution,} first
 for a continuous obstacle. Under assumptions on the time derivative of the obstacle, the existence of a variational solution is treated in \cite{altl83} and \cite{bogeleindm09}. See also \cite{kinderlehrers80}.

Let $\psi \in
C(\overline \Om_T)$
and define the class $\mathcal F_{\psi}$  consisting of all functions
$v\in C(\overline \Om_T)$ such that
\[
v\in L^p(0,T; W^{1,p}(\Om)),\quad v=\psi\textrm{ on } \partial_p \Om_T\quad
\textrm{and} \quad v\geq \psi  \textrm{ in } \Om_T.
\]
\begin{defi}
\label{def:smooth}
A function  $v\in
\mathcal F_{\psi}$ is a \emph{variational solution} to the obstacle problem if
\begin{equation}
\label{eq:smooth}
\begin{split}
\int\!\!\int_{\Om_T} &\bigg(\abs{\nabla v}^ {p-2} \nabla v \cdot \nabla
(\phi-v)+(\phi-v) \parts{\phi}{t}\bigg) \ud x \ud t \\
&\geq \half \int_\Om
\abs{\phi(x,T)-v(x,T)}^2 \ud x
\end{split}
\end{equation}
for all $\phi\in C^{\infty}(\Om_T)$ in $\mathcal F_{\psi}$ such that $\parts{\phi}{t} \in L^{q}(\Om_T)$,\ $q=p/(p-1)$.
\end{defi}

By an approximation procedure, we can extend the
admissible class of test functions to include all continuous
$\phi\in L^p(0,T;W^{1,p}(\Om))$ in $\mathcal F_{\psi}$ such that
$\parts{\phi}{t} \in L^{q}(\Om_T)$,\ $q=p/(p-1)$.
\medskip

For a smooth variational solution $v$,  integration by parts implies
\[
\begin{split}
\int_0^T\!\!\int_\Om (\phi-v) \parts{\phi}{t} \ud x \ud t&=\half \int_\Om
\abs{\phi(x,T)-v(x,T)}^2 \ud x\\
&\hspace{1 em}+\int_0^T\!\!\int_\Om (\phi-v) \parts{v}{t}  \ud x \ud t
\end{split}
\]
and thus \eqref{eq:smooth} can be written as
\begin{equation}
\label{eq:obstacle-problem-for-smooth}
\begin{split}
\int\!\!\int_{\Om_T} \left(\abs{\nabla v}^ {p-2} \nabla v \cdot \nabla
(\phi-v)+(\phi-v) \parts{v}{t}\right)  \ud x \ud t\geq 0.
\end{split}
\end{equation}

 Next we show that the least solution satisfies \defiref{def:smooth},
 and thus, for a continuous obstacle,
 this gives us the existence of a variational solution.

Below, we use the standard mollification
\begin{equation}
\label{eq:friedrichs}
\begin{split}
u_\sigma(x,t)=\int_\R u(x,t-s) \zeta_\sigma(s)\ud s
\end{split}
\end{equation}
with Friedrichs' mollifier
\[
\begin{split}
\zeta_\sigma(s)=\begin{cases}
\frac{C}{\sigma} e^{-\sigma^2/(\sigma^2-s^2)},& \abs{s}<\sigma\\
0,& \abs{s}\geq \sigma,
\end{cases}
\end{split}
\]
where the constant $C$ is chosen so that $\int_{-\infty}^{\infty} \zeta_\sigma(s)\ud s=1$.
Let $\vp \in C^\infty_0(\Om_T),\ \vp \geq 0$ and choose
$\sigma<\dist\left(\supp(\vp),\Om\times\{0,T\}\right)$. We insert
$\vp_\sigma$ into \eqref{eq:weak-super}, change variables, and apply
Fubini's theorem to obtain
\begin{equation}
\label{eq:friedrichs-mollified}
\int\!\!\int_{\Om_T}\left(\left(\abs{\nabla u}^{p-2} \nabla u\right)_\sigma \cdot \nabla
\vp +\parts{u_\sigma}{t}\vp\right)\ud x \ud t\geq 0
\end{equation}
for the weak supersolution $u$.
The analogous formula with equality  holds for weak solutions.
\begin{theo}
\label{thm:from-potential-to-variational}
Let $\psi \in C_0( \Om_T).$
 Then
the least solution $w^*$ is also a variational solution. In other words, $w^*$  satisfies the variational inequality
\[
\begin{split}
\int\!\!\int_{\Om_T} &\left(\abs{\nabla w^*}^ {p-2} \nabla w^* \cdot \nabla
(\phi-w^*)+(\phi-w^*) \parts{\phi}{t}\right) \ud x \ud t \\
&\geq \half \int_\Om
\abs{\phi(x,T)-w^*(x,T)}^2 \ud x
\end{split}
\]
for all $\phi\in C^{\infty}(\Om_T)$ in $\mathcal F_\psi$ such that $\parts{\phi}{t} \in L^{q}(\Om_T),\ q=p/(p-1)$.
\end{theo}
\begin{proof}

First, observe  that $w^*=\psi $ on $\partial_p \Om_T$ by \theoref{thm:continuous}, and $w^*\in L^p(0,T;W^{1,p}_0(\Om))$, cf.\  Lemma~\ref{lem:global-int}. Denote by $\chi^{h}_{0,T}$ a continuous, piecewise linear approximation of a characteristic function such that
\begin{equation}
\label{eq:cut-off}
\begin{split}
\begin{cases}
      \parts{\chi^{h}_{0,T}}{t}=1/h, &  \text{ if }h<t<2h,\\
      \chi^{h}_{0,T}=1, & \text{ if } 2h<t<T-2h,\\
      \parts{\chi^{h}_{0,T}}{t}=-1/h, &  \text{ if }T-2h<t <T-h,\\
      \chi^{h}_{0,T}=0, & \text{ otherwise},
    \end{cases}
\end{split}
\end{equation}
and let $\phi$ be the test function in the theorem. Then an
approximation argument justifies the use of $$ \varphi =\chi^{h}_{0,T}(\phi_\sigma-w^*_\sigma)_+=\chi^{h}_{0,T}\max(\phi_\sigma-w^*_\sigma\, , \,0)$$ as a test function in \eqref{eq:friedrichs-mollified}, so that
\[
\begin{split}
\int\!\!\int_{\Om_T}\bigg(\Big(\abs{\nabla w^*}^{p-2} \nabla w^*\Big)_\sigma \cdot \chi^{h}_{0,T}&\nabla
(\phi_\sigma-w^*_\sigma)_+ \\
&+ \parts{w^*_\sigma}{t}\chi^{h}_{0,T}(\phi_\sigma-w^*_\sigma)_+ \bigg)\ud x \ud t \geq 0.
\end{split}
\]
By adding the integral of $-\parts{\phi_\sigma}{t}\chi^{h}_{0,T}(\phi_\sigma-w^*_\sigma )_+$ to both sides and integrating by parts, we get
\[
\begin{split}
\int\!\!\int_{\Om_T}&\bigg(\Big(\abs{\nabla w^*}^{p-2} \nabla w^*\Big)_\sigma \cdot \chi^{h}_{0,T} \nabla
(\phi_\sigma-w^*_\sigma)_+ \\
&\hspace{8 em}+\frac{1}{2}((\phi_\sigma-w^*_\sigma)_+)^2\parts{\chi^{h}_{0,T}}{t}\bigg)\ud x \ud t \\
&\geq -\int\!\!\int_{\Om_T}\parts{\phi_\sigma}{t}\chi^{h}_{0,T}(\phi_\sigma-w^*_\sigma )_+\ud x \ud t.
\end{split}
\]

Letting first $\sigma\to 0$ and then $h\to 0$, we get
\begin{equation}
\label{eq:above-obstacle}
\begin{split}
\int\!\!\int_{\Om_T}\bigg(\abs{\nabla w^*}^{p-2} \nabla w^* &\cdot \nabla
(\phi-w^*)_+ +\parts{\phi}{t}(\phi-w^* )_+\bigg)\ud x \ud t \\
\geq &\frac{1}{2}\int_{\Om}(\phi(x,T)-w^*(x,T))_+^2\ud x.
\end{split}
\end{equation}

Next we perform  a similar calculation, using the fact  that $w^*$ is
$p$-parabolic in  the open set $U=\Om_T\cap \{\phi<w^*\}$. This time we use the
test function $\chi^{h}_{0,T}(\phi_\sigma-w^*_\sigma)_-=\chi^{h}_{0,T}\min(\phi_\sigma-w^*_\sigma,0)$. Since $\phi$ is smooth, we can choose a decreasing sequence of smooth functions  $\phi^i$ converging to $\phi$ so that
  \[
  \begin{split}
  \{\phi^i-w^*<0\}\Subset U.
  \end{split}
  \]
  For a fixed index $i$, we can choose  $\sigma>0$ so small that also
  \[
  \begin{split}
    \{(\phi^i-w^*)_\sigma<0\}\Subset U.
  \end{split}
  \]
  A similar calculation as the previous one implies, since $w^*$ is $p$-parabolic  in $U$,
\[
\begin{split}
\int_{U}&\bigg(\Big(\abs{\nabla w^*}^{p-2} \nabla w^*\Big)_\sigma \cdot \chi^{h}_{0,T} \nabla
(\phi^i_\sigma-w^*_\sigma)_- \\
&\hspace{8 em}+\frac{1}{2}((\phi^i_\sigma-w^*_\sigma)_-)^2\parts{\chi^{h}_{0,T}}{t}\bigg)\ud x \ud t \\
&=-\int_{U}\parts{\phi^i_\sigma}{t}\chi^{h}_{0,T}(\phi^i_\sigma-w^*_\sigma )_-\ud x \ud t.
\end{split}
\]
As first $\sigma\to 0$, then $h\to 0$ and finally $i\to \infty$, we
obtain
\begin{equation}
\label{eq:below-obstacle}
\begin{split}
\int_{U}&\bigg(\abs{\nabla w^*}^{p-2} \nabla w^* \cdot \nabla
(\phi-w^*)_- +\parts{\phi}{t}(\phi-w^*)_-\bigg)\ud x \ud t \\
&= \frac{1}{2}\int_{\Om}(\phi(x,T)-w^*(x,T))_-^2\ud x.
\end{split}
\end{equation}
Together \eqref{eq:above-obstacle} and \eqref{eq:below-obstacle} prove the claim.
\end{proof}

 We recall  the convenient convolution
\begin{equation}
\label{eq:naumann-convolution}
\begin{split}
u_\eps(x,t)=\frac{1}{\eps}\int_0^t e^{(s-t)/\eps} u(x,s) \ud s,
\end{split}
\end{equation}
which is expedient for our purpose;
see for example \cite{naumann84}, \cite{boccardodgo97}, and
\cite{kinnunenl06}.
It has the following properties.
\begin{lem}
\label{lem:naumann}
\begin{enumerate}[(i)]
\item \label{item:time-derivative-of-naumann}
If $u \in L^p(\Om_T)$, then
\[
\begin{split}
\norm{u_\eps}_{L^p(\Om_T)}\leq \norm{u}_{L^p(\Om_T)},
\end{split}
\]
\[
\begin{split}
\parts{u_\eps}{t}=\frac{u-u_\eps}{\eps}\in L^p(\Om_T),
\end{split}
\]
and \[
\begin{split}
u_\eps\to u \quad \trm{in} \quad L^p(\Om_T)\quad \trm{as} \quad \eps\to 0.
\end{split}
\]
\item If $\nabla u\in L^p(\Om_T)$, then $\nabla u_\eps=(\nabla u)_\eps$ componentwise,
\[
\begin{split}
\norm{\nabla u_\eps}_{L^p(\Om_T)}\leq \norm{\nabla u}_{L^p(\Om_T)},
\end{split}
\]
and
\[
\begin{split}
\nabla u_\eps\to \nabla u \quad \trm{in} \quad L^p(\Om_T)\quad \trm{as} \quad \eps\to 0.
\end{split}
\]
\item
Furthermore, if $u^k\to u$ in $L^p(\Om_T)$, then also
\[
\begin{split}
u^k_\eps\to u_\eps\quad \trm{and}\quad \parts{u^k_\eps}{t}\to \parts{u_\eps}{t}
\end{split}
\]
in $L^p(\Om_T)$.
\item If $\nabla u^k\to \nabla u$ in $L^p(\Om_T)$, then $\nabla u^k_\eps\to \nabla u_\eps$ in $L^p(\Om_T)$.
\item Analogous results hold for the weak convergence in $L^p(\Om_T)$.
\item Finally, if $\vp\in C(\ol \Om_T)$, then
\[
\begin{split}
\vp_\eps(x,t) +e^{- \frac{t}{\eps}} \vp(x,0) \to \vp(x,t)
\end{split}
\]
\emph{uniformly} in $\Om_T$ as $\eps\to 0$.
\end{enumerate}
\end{lem}

Next we show that a variational solution is unique for a continuous compactly supported obstacle.
\begin{theo}
\label{thm:smooth-unique}
Let $\psi \in C_0(\Om_T)$.  The variational solution in \defiref{def:smooth} with this obstacle  is unique.
\end{theo}
\begin{proof}
Suppose that $u$ and $v$ are two solutions. They are continuous. We sum up
\[
\begin{split}
\int\!\!\int_{\Om_T} &\bigg(\abs{\nabla u}^ {p-2} \nabla u \cdot \nabla
(\phi-u)+(\phi-u) \parts{\phi}{t}\bigg) \ud x \ud t \\
&\geq \half \int_\Om
\abs{\phi(x,T)-u(x,T)}^2 \ud x
\end{split}
\]
and
\[
\begin{split}
\int\!\!\int_{\Om_T} &\bigg(\abs{\nabla v}^ {p-2} \nabla v \cdot \nabla
(\phi-v)+(\phi-v) \parts{\phi}{t}\bigg) \ud x \ud t \\
&\geq \half \int_\Om
\abs{\phi(x,T)-v(x,T)}^2 \ud x.
\end{split}
\]
We end up with
\begin{gather}
\label{eq:sum-in-uniqueness-proof}
\int\!\!\int_{\Om_T} \Big(\abs{\nabla v}^ {p-2} \nabla v \cdot \nabla (v-\phi)-\abs{\nabla u}^ {p-2} \nabla u \cdot \nabla
(\phi-u)\Big) \ud x \ud t \nonumber \\
\leq 2  \int\!\!\int_{\Om_T} \left(\phi-\frac{u+v}{2}\right) \parts{\phi}{t} \ud x \ud t.
\end{gather}
 If we could choose the test function $\phi$ equal to $(u+v)/2$, the
 desired result would follow easily from the structure of the
 left-hand member. However, this function is not admissible, since its
 time derivative is not guaranteed. We modify it by utilizing convolution
 \eqref{eq:naumann-convolution}, and use the test function
\[
\begin{split}
\phi=\left(\frac{u+v}{2}+ \alpha \eta(x)\right)_{\eps} ,
\end{split}
\]
where $\eta \in C_{0}^{\infty}(\Omega)\,\, \eta \geq 0$ and  $\eta =1$
on $\spt \psi$. Here $\alpha > 0$ is given and $0 < \eps <
\eps(\alpha)$, where $\eps(\alpha)$
is so small that
$$ \phi \geq\left( \psi+\alpha \eta \right)_{\eps}\geq \psi$$
in $\Omega_T.$
 Now
$$\frac{\partial \phi}{\partial t} =
\frac{1}{\eps}\left[\left(\frac{u+v}{2}+\alpha \eta\right)- \left(
    \frac{u+v}{2}+\alpha \eta\right)_{\eps}\right]$$
and so we obtain
\[
\begin{split}
 \int&\!\!\int_{\Om_T}\left(\phi-\frac{u+v}{2}\right)\! \parts{\phi}{t} \ud x \ud t\\
&= \int\!\!\int_{\Om_T}\left(\phi-\left(\frac{u+v}{2}+\alpha \eta\right)\right)\! \parts{\phi}{t} \ud x \ud t+\alpha\int\!\!\int_{\Om_T}  \eta  \parts{\phi}{t} \ud x \ud t\\
&= - \frac{1}{\eps} \int\!\!\int_{\Om_T} \! \left[\left(\frac{u+v}{2}+\alpha \eta\right)- \left(
    \frac{u+v}{2}+\alpha \eta\right)_{\eps}\right]^{2}\! \ud x \ud t
\\
&\hspace{13 em}+ \alpha  \int\!\!\int_{\Om_T} \eta(x)\frac{\partial \phi}{\partial t}
\ud x \ud t\\
&\leq 0 + \alpha
\int_{\Omega}\eta (x)\left(\frac{u+v}{2}+\alpha\eta \right)_{\eps}\!(x,T)\,dx.
\end{split}
\]
Now we can safely let $\eps \rightarrow 0$ after which we also let
$\alpha \rightarrow 0.$
The result is that
$$\frac{1}{2}  \int\!\!\int_{\Om_T}(|\nabla v|^{p-2}\nabla v - |\nabla
u|^{p-2}\nabla u)\cdot (\nabla v - \nabla u)  \ud x \ud t \leq 0.$$
The integrand is non-negative and zero only for $\nabla v = \nabla u$. Since $u$ and
$v$ have the same boundary values, they coincide.
\end{proof}

\begin{cor}
\label{cor:smooth-variational=least}
For the obstacle $\psi \in C_0(\Omega_T)$, the
 variational solution coincides with the least solution.  In particular, the variational solution is a weak supersolution.
\end{cor}

\begin{proof} According to Theorem 3.7 the least solution $w^*$ is also a
  variational solution. But there is only one variational solution
  according to the theorem.
\end{proof}
The corollary can be modified to include the case $\psi \in C^\infty(\ol \Om_T)$.  For a different approach to a continuous obstacle problem, see \cite{korteks09}.

\begin{cor}
\label{cor:continuous-ordered-obstacles}
Let $v_1,\ v_2$ be the variational solutions with the obstacles
 $\psi_1,\psi_2 \in C_0(\Omega_T)$. If $\psi_1 \leq
 \psi_2$, then $v_1 \leq v_2.$
\end{cor}

\begin{proof} By the previous corollary they are the least solutions:
  $v_1 = w_1^*$ and  $v_2 = w_2^*$. By Theorem 2.8 these
  are weak supersolutions. Since $v_2 \geq \psi_2 \geq \psi_1$, we
  must have $ w_1^* \leq v_2$, as $ w_1^*$ is the least
  one.
\end{proof}

\section{Irregular obstacle}

In this section we treat the irregular obstacle
with
 \begin{align*}
\text{\textsf{Assumption:}}&\quad\psi \in
L^p(0,T;W^{1,p}(\Omega)),\\ &\quad \spt\psi\Subset\Omega_T,\quad 0
\leq \psi \leq L.
\end{align*}The simplifying effect of the compactness assumption is not fully
utilized: the benefit for us comes from the zero region near the lateral
boundary $\partial \Omega \times [0,T].$

The least solutions are  well defined in this generality, but there is
a difficulty. On the one hand, the variational definition fails to
guarantee uniqueness, if only smooth test functions are admissible,
see Section~\ref{sec:special-cases}. On the other hand, complications
with time derivatives prevent us from using all the test functions
from the regularity class  the obstacle belongs to. Nevertheless, an approximation with variational solutions with suitable smooth obstacles turns out to give exactly the unique least solution, Theorem~\ref{thm:uniqueness}.

 However, first we
 discuss a
 convergence result in the elliptic theory, Proposition~\ref{prop:elliptic-approximation}.
The parabolic counterpart to the proposition is not a simple one.

 For
   $\psi\in W^{1,p}(\Om)$, we define the class
\[
\begin{split}
\mathcal K_{\psi}=\{\phi\in W^{1,p}(\Om)\,:\,\phi\geq \psi \trm{ a.e.\ in } \Om,\, \phi-\psi\in W^{1,p}_0(\Om)\}.
\end{split}
\]
Then $v \in \mathcal K_{\psi}$ is a variational solution to the
elliptic obstacle problem, if
\begin{equation}
\label{eq:elliptic}
\int_{\Om} \abs{\nabla v}^ {p-2} \nabla v \cdot \nabla
(\phi-v) \ud x \geq 0
\end{equation}
for every $\phi\in \mathcal K_{\psi}$. The variational solution agrees
with the least solution: $v=w^*$ a.e.\ in this case, see for example
\cite[Theorem 9.26.]{heinonenkm93}.  Our approximative definition
coincides with the least solution in the elliptic case. Notice that we
do not demand $\phi$ to be continuous now. The approximants are pretty
arbitrary in the next proposition.
\begin{prop}[Elliptic case]
\label{prop:elliptic-approximation}
Let $v_{\psi_{j}} \in  \mathcal K_{\psi_{j}}$ denote the variational
solution with the obstacle $\psi_{j}$. If $\psi_{j} \to \psi$ in
$W^{1,p}(\Omega),$ then
$$v_{\psi_{j}} \to v_{\psi}\quad \trm{in}\quad W^{1,p}(\Omega),$$
 where  $v_{\psi}$ is the variational solution with $\psi$ as an obstacle.
\end{prop}

\begin{proof} Use the test functions\footnote{Such a test function is
    out of the question in the parabolic case, because of
    complications with the time derivative.}
$$\phi_{j} = v_{\psi} + \psi_{j} - \psi \in  \mathcal K_{\psi_{j}},\quad \phi
= v_{\psi_{j}} + \psi - \psi_{j} \in \mathcal K_{\psi}$$
to prove this. See also Theorem 1.4 in Li--Martio \cite{lim94}.
\end{proof}

Let us leave the elliptic case and return to the parabolic situation.

\begin{lemma}\label{lem:global-int}
 Let $\psi\in L^p(0,T;W^{1,p}(\Om)),\ \spt\psi\Subset \Om_T,\ 0  \leq
\psi\leq L$, and let $w^*$ be the least solution with the obstacle  $\psi$. Then
$w^*$ is $p$-parabolic  in $\Om_T\setminus \spt \psi$ and  $w^* \in L^p(0,T;W^{1,p}_0(\Om))$.
\end{lemma}

\begin{proof}
 The first part of the proof is similar to the end of the proof of \theoref{thm:continuous}

To prove the \emph{global} integrability of $w^*$, we show that $w^*$
coincides with the solution to a boundary value problem near the
lateral boundary. To this end, we choose a smooth open set $D\subset
\Rn$ such that $\spt \psi\Subset D\times(t_1,t_2)$.  We solve the Evolutionary p-Laplace Equation \eqref{eq:weak-sol} in $(\Om\setminus \ol D) \times(0,T)$ with the boundary values
\[
\begin{split}
\begin{cases}
u&=w^* \quad \trm{on}\quad \partial D\times(0,T)\\
u&=0 \quad \trm{on}\quad  (\Om\setminus \ol D)\times\{0\}\\
u&=0   \quad \trm{on}\quad \partial \Om \times(0,T).
\end{cases}
\end{split}
\]
The continuity of $u$ and $w^*$ in $(\ol \Om\setminus D) \times(0,T)$ and the "elliptic" comparison principle, Proposition 3 in \cite{lindqvistm07} or Lemma 4.5 in \cite{kortekp10}, imply that the set  $\{u>w^*+\eps\}$ is empty for any $\eps>0$. Thus $u\le w^*+\eps$, and since $\eps>0$ was arbitrary, it follows that
\[
\begin{split}
u= w^* \quad \trm{in}\quad (\Om\setminus \ol D)\times(0,T).
\end{split}
\]
This implies the claim.
\end{proof}

Below we will use the averaged inequality  with
the convolution \eqref{eq:naumann-convolution}, cf. \cite{kinnunenl06}.
The averaged equation for a weak  supersolution $u$ in $\Om_T$ is the following
\begin{equation}
\label{eq:mollified-KL}
\begin{split}
\int\!\!\int_{\Om_T}&\bigg(\Big(\abs{\nabla u}^{p-2} \nabla u\Big)_\eps \cdot \nabla
\vp -u_\eps\parts{\vp}{t}\bigg)\ud x \ud t\\
&\hspace{5 em}+\int_\Om u_\eps(x,T) \vp (x,T) \ud x\\
& \geq \int_\Om u(x,0) \left(\frac{1}{\eps}\int_0^T \vp(x,s) e^{-s/\eps} \ud s \right) \ud x
\end{split}
\end{equation}
valid for all test functions $\vp \geq 0$ vanishing on the parabolic boundary $\partial_p \Om_T$. To see this, we observe that the definition of a supersolution gives us
\[
\begin{split}
\int_s^T\!\!\int_{\Om}& \bigg(\abs{\nabla u(x,t-s)}^{p-2} \nabla u(x,t-s) \cdot \nabla
\vp(x,t)\\
& \hspace{1 em}-u(x,t-s)\parts{\vp}{t}(x,t)\bigg)\ud x \ud t\ +\ \int_\Om u(x,T-s) \vp (x,T) \ud x\\
&\geq \int_\Om u(x,0) \vp (x,s) \ud x,
\end{split}
\]
when $0\leq s\leq T$. Notice that $(x,t-s)\in \ol \Om_T$. To obtain \eqref{eq:mollified-KL} we multiply the above inequality by $e^{-s/\eps}/\eps$, integrate over $[0,T]$ with respect to $s$, and finally change the order of integration to obtain.
Upon integration by parts we see that for a supersolution $u\in L^p(0,T;W^{1,p}(\Om))$  inequality \eqref{eq:mollified-KL} implies
\begin{equation}
\label{eq:mollified}
\begin{split}
\int\!\!\int_{\Om_T}&\bigg(\Big(\abs{\nabla u}^{p-2} \nabla u\Big)_\eps \cdot \nabla
\vp +\parts{u_\eps}{t}\vp\bigg)\ud x \ud t\\
& \geq \int_\Om u(x,0) \left(\frac{1}{\eps}\int_0^T \vp(x,s) e^{-s/\eps} \ud s \right) \ud x
\end{split}
\end{equation}
for every $\vp \in  C(\ol \Om_T)\cap C^{\infty}({\Om_T}),\, \vp\geq
0$,
vanishing on the parabolic boundary $\partial_p \Om_T$.

We will use only the simpler version
\begin{equation}
\label{eq:naumann-zero-rhs}
\int\!\!\int_{\Om_T}\bigg(\Big(\abs{\nabla u}^{p-2} \nabla u\Big)_\eps \cdot \nabla
\vp +\parts{u_\eps}{t}\vp\bigg)\ud x \ud t \geq 0
\end{equation}
valid for $u \geq 0$ and $\vp$ vanishing on $\partial_p \Om_T$.

By approximating an irregular obstacle $\psi$ by the  mollified obstacles
 $\psi_{\eps}$ and solving the corresponding variational  problems, we
 arrive at the least solution as a limit. This is the content of
Theorem~\ref{thm:uniqueness}. However, arbitrary smooth
 approximations to the obstacle will not work; we use convolutions.
 The key observation in the proof of Theorem~\ref{thm:uniqueness} is
that we can,  without affecting the limit of the approximation,
 replace the obstacle by the least supersolution above the
 obstacle. We start with an auxiliary result.
\begin{lem}
\label{lem:ordered-obstacles}
Suppose that $\psi^u,\psi^v\in L^p(0,T;W^{1,p}_0(\Om))$ and define $\psi^u_\eps, \psi^v_\eps$ as in formula \eqref{eq:naumann-convolution}.
Let $u$ and  $v$ be the variational solutions with $\psi^u_\eps$ and $\psi^v_\eps$. If $\psi^u_\eps\geq \psi^v_\eps$, then $u\geq v$ almost everywhere.
\end{lem}
\begin{proof}
   First we extend $\psi^u$
  and $\psi^v$ by zero outside $\Om$. Then we mollify the
  obstacles $\psi^u_\eps$ and $\psi^v_\eps$ in space using the
  standard Friedrichs' mollifier with parameter $\sigma$.

\sloppy
We solve the variational obstacle problem in $\Om\times(0,T)$ with $\psi^u_{\eps,\sigma},\psi^v_{\eps,\sigma} \in C^\infty(\ol \Om_T)$.
Since the obstacles are smooth and ordered, we conclude from Corollary 3.16 that  $u^\sigma, v^\sigma$ are weak supersolutions and
\begin{equation}
\label{eq:approximants-in-order}
\begin{split}
  v^\sigma\leq u^\sigma
\end{split}
\end{equation} almost everywhere. The corollary is formulated for $C_0$-obstacles, but it can be modified to the present setting as well.
Alternatively, according to \cite{altl83}, \cite{bogeleindm09}, variational solutions $u^\sigma$
$v^\sigma$ exist, attain the boundary values  in
$L^p(0,T;W^{1,p}_0(\Om))$  prescribed by the obstacles, and have time derivatives in the dual space.  Thus $u^\sigma ,v^\sigma$ turn out to be supersolutions, and we can use $u^\sigma+(v^\sigma-u^\sigma)_+$ as a test function
for $u^\sigma$ and $v^\sigma-(v^\sigma-u^\sigma)_+$ for
$v^\sigma$ to deduce the same result.

Next we establish the needed convergence results. Observe that
\begin{equation}
\label{eq:variational-inequality}
\begin{split}
\int\!\!\int_{\Om_T} \Big(\abs{\nabla u^\sigma}^{p-2} \nabla u^\sigma &\cdot \nabla
(\psi^u_{\eps,\sigma}-u^\sigma)\\
&+(\psi^u_{\eps,\sigma}-u^\sigma) \parts{\psi^u_{\eps,\sigma}}{t}\Big) \ud x \ud t\geq 0
\end{split}
\end{equation}
gives us the global estimate
\[
\begin{split}
\int\!\!\int_{\Om_T} \abs{\nabla u^\sigma}^p \ud x \ud t \le C \int\!\!\int_{\Om_T} \abs{\nabla \psi^u_{\eps,\sigma}}^p  \ud x \ud t+  C \int\!\!\int_{\Om_T} \abs{\parts{\psi^u_{\eps,\sigma}}{t}}\! \ud x \ud t.
\end{split}
\]
 This uniform bound with respect to $\sigma$  implies that a
 subsequence of  $u^\sigma$  converges weakly in $L^p(0,T;W^{1,p}(\Om))$ to some limit $\tilde u$. Furthermore, \theoref{thm:superparab-compactness} gives
us a pointwise convergence of $u^\sigma$ and $\nabla u^\sigma$ to
$\tilde u$ and $\nabla  \tilde u$. This is enough to pass to a limit
under the integral sign in \eqref{eq:variational-inequality}. It follows that $\tilde u$ is a weak supersolution.

 Since $\psi^u_{\eps,\sigma}-u^\sigma\in L^p(0,T;W^{1,p}_0(\Om))$ we deduce that
\[
\begin{split}
\psi^u_{\eps}-\tilde u\in L^p(0,T;W^{1,p}_0(\Om)).
\end{split}
\]
 This is enough for using the uniqueness from Theorem 6.1 in
 \cite{bogeleindm09} to conclude that $\tilde u$ is the unique
 variational solution  with the obstacle $\psi^\eps_u$. In other words
   $\tilde u = u.$ We complete the proof by combining this result and \eqref{eq:approximants-in-order}.
\end{proof}

The previous proof contains the following result.
\begin{corollary}
\label{cor:supersol}
 Let $\psi\in L^p(0,T;W^{1,p}(\Omega))$ and define $\psi_\eps$ as in formula \eqref{eq:naumann-convolution}. Then the variational solution $u$ with the obstacle $\psi_\eps$ is a supersolution.
\end{corollary}

The next theorem shows that, if the obstacle itself is  a
 supersolution, then the approximation gives the same supersolution at
 the limit.

\begin{theo}
\label{thm:supersolution-obstacle}
Let $w\in L^p(0,T;W^{1,p}(\Om)),\ 0\leq w\leq L,$  be a weak supersolution
and define $w_\eps$ as in formula \eqref{eq:naumann-convolution}.  Let $v^\eps$ be the variational solutions with the mollified obstacles  $w_\eps$. Then, passing to a subsequence if necessary,
\[
\begin{split}
\nabla v^\eps \rightarrow \nabla w \quad \text{in} \quad L^p(\Om_T),
\end{split}
\]
$$v^\eps \to w, \quad \nabla v^\eps \to \nabla w \quad \textrm{a.e.\ in}\quad
 \Omega_T.$$
\end{theo}
\begin{proof}

 By \corref{cor:supersol}, $v^\eps$ is a weak supersolution and further $0\le v^\eps\le L$.
According to \theoref{thm:superparab-compactness}, there exists a subsequence, still denoted by $v^\eps$, and a limit $v$ such that
$$v^\eps \to v, \quad \nabla v^\eps \to \nabla v\quad \textrm{a.e.\ in}\quad
 \Omega_T.$$
Thus we have  to show that $v=w$ almost everywhere. To this end,
 observe that the obstacle $w_\eps$ is an admissible test function for $v^\eps$ and write
\[
\begin{split}
\int\!\!\int_{\Om_T} &\left(\abs{\nabla v^\eps}^{p-2} \nabla v^\eps \cdot \nabla
(w_\eps-v^\eps)+(w_\eps-v^\eps) \parts{w_\eps}{t}\right) \ud x \ud t \geq 0.
\end{split}
\]
On the other hand, since $w \geq 0$ is a weak supersolution and $v^\eps\geq w_\eps$, we have by \eqref{eq:naumann-zero-rhs} that
\[
\begin{split}
\int\!\!\int_{\Om_T}\Big(\left(\abs{\nabla w}^{p-2} \nabla w\right)_\eps & \cdot \nabla
(v^\eps-w_\eps)+(v^\eps-w_\eps) \parts{w_\eps}{t}\Big) \ud x \ud t\geq0.
\end{split}.
\]
Since
 $v^\eps$ takes the boundary values on the parabolic boundary $\partial_p \Om_T$ in a suitable sense  an approximation argument justifies our use of $v^\eps-w_\eps$ as a test function in \eqref{eq:naumann-zero-rhs}.

We sum up the inequalities to obtain
\begin{equation}
\label{eq:key-estimate}
\begin{split}
\int\!\!\int_{\Om_T} &\left(\abs{\nabla v^\eps}^{p-2} \nabla v^\eps -\left(\abs{\nabla w}^{p-2} \nabla w\right)_\eps \right)\cdot\nabla
(w_\eps-v^\eps) \ud x \ud t \geq 0.
\end{split}
\end{equation}

Next we aim at passing to the limit under the integral sign  in order
to deduce that $v^\eps\to w$. We write
\[
\begin{split}
\int\!\!\int_{\Om_T}& \left(\abs{\nabla v^\eps}^{p-2} \nabla v^\eps - \abs{\nabla w_\eps}^{p-2} \nabla w_\eps\right)\cdot\nabla
(v^\eps-w_\eps) \ud x \ud t \\
&\leq
\int\!\!\int_{\Om_T} \left(\left(\abs{\nabla w}^{p-2} \nabla w\right)_\eps
- \abs{\nabla w_\eps}^{p-2} \nabla w_\eps \right)\cdot\nabla
(v^\eps - w_\eps) \ud x \ud t \\
& \leq
\frac{\alpha^p}{p}\int\!\!\int_{\Om_T} |\nabla(v^\eps - w_\eps)|^{p}
\ud x \ud t\\
&\hspace{1 em}+\frac{1}{q\alpha^q}\int\!\!\int_{\Om_T} |\left(|\nabla w|^{p-2} \nabla w\right)_\eps
-  |\nabla w_\eps|^{p-2}\nabla w_\eps|^{q} \ud x \ud t,
\end{split}
\]
where Young's inequality was used for $\alpha > 0$ and $q = p/(p-1).$
The integrand in the left-hand side is greater than
$$2^{2-p}|\nabla(v^\eps - w_\eps)|^{p}$$ and we fix $\alpha$ so small
that the integral of this minorant can absorb the first integral on
the right-hand side. In other words
\[
\begin{split}
\int\!\!\int_{\Om_T}& |\nabla(v^\eps - w_\eps)|^{p}
\ud x \ud t\\
& \leq
C(p)\int\!\!\int_{\Om_T} |\left(|\nabla w|^{p-2} \nabla w\right)_\eps
-  |\nabla w_\eps|^{p-2}\nabla w_\eps|^{q} \ud x \ud t.
\end{split}
\]
As $\eps \rightarrow 0$ the majorant vanishes and we arrive at
\begin{equation}
\int\!\!\int_{\Om_T} |\nabla(v - w)|^{p}
\ud x \ud t \leq \lim_{\eps\to0} \int\!\!\int_{\Om_T} |\nabla(v^\eps - w_\eps)|^{p}
\ud x \ud t = 0,
\end{equation}
where Fatou's lemma was used.

It follows  that $\nabla v=\nabla w$ a.e. in $\Om_T$. We assure that $w-v\in L^p(0,T;W^{1,p}_0(\Om))$  similarly as at the end of the proof of \lemref{lem:ordered-obstacles}, and the proof is complete.
\end{proof}

From the previous theorem we can deduce that the variational solutions with the mollified obstacles converge to the least solution.
\begin{theo}
\label{thm:uniqueness}
Let $\psi\in L^p(0,T;W^{1,p}(\Om)),\ \spt\psi\Subset \Om_T,\ 0  \leq
\psi\leq L$, and let $u^\eps$ be the variational solutions with the mollified obstacles $\psi_\eps$.
Let $w^*$ denote the least solution with the obstacle $\psi$.
 Then
$$u^\eps \to w^*, \quad \nabla u^\eps \to \nabla w^*\quad \textrm{a.e.\ in}\quad
\Omega_T.$$
\end{theo}
\begin{proof}
 By \corref{cor:supersol}, $u^\eps$ is a weak supersolution and $0\le u^\eps\le L$. \theoref{thm:superparab-compactness} yields a subsequence, still denoted by $u^\eps$, and a limit $u$ such that
$$u^\eps \to u, \quad \nabla u^\eps \to \nabla u\quad \textrm{a.e.\ in}\quad
 \Omega_T$$
as $\eps\to 0$. The function $u$ is a weak supersolution, and we
may even assume it to be $\essliminf$-regularized.
Since $\psi_{\eps} \to \psi$, $u \geq \psi$ almost everywhere, and so we
conclude that
\[
\begin{split}
w^*\leq u ,
\end{split}
\]
because $w^*$ is the least solution.

Let $v^\eps$ be the variational solutions with the mollified obstacles
$w^*_\eps$. Since $w^* \geq \psi,$ also $w^*_{\eps} \geq \psi_{\eps}.$
Due to the assumption $\spt\psi\subset \Om_T$, we see by
Lemma~\ref{lem:global-int} that $w^*\in
L^p(0,T;W^{1,p}_0(\Om))$. By the previous lemma
$$v^\eps \to w^*, \quad \nabla v^\eps \to \nabla w^*\quad \textrm{a.e.\ in}\quad
\Omega_T$$ as $\eps\to 0$, at least for a subsequence.  But now
$w^*_\eps\geq \psi_\eps$ implies that $v^\eps\geq {u^\eps}$ almost
everywhere according to \lemref{lem:ordered-obstacles}.  Thus by
passing to a limit, we have
\[
\begin{split}
w^*\geq u
\end{split}
\]
almost everywhere.
Thus $u=w^*$ almost everywhere.
\end{proof}

We could also have taken a slightly different approach, and used the mollification \eqref{eq:naumann-convolution} in
time and then a mollification analogous to \eqref{eq:friedrichs} in space.
The space mollifications are well defined also near the lateral boundary as we extend the functions by zero outside $\Om$. A good point in this approach is that, since the mollified
obstacles are in $C^\infty$, \lemref{lem:ordered-obstacles} is
immediate. Observe also that, in this approach, we do not assume that the obstacle is in the Sobolev space. Thus for example a characteristic function is an admissible  obstacle.
\begin{theo}
\label{thm:uniqueness-space}
Let $\psi$ be a measurable function such that $\spt\psi\Subset\Om_T,\ 0  \leq
\psi\leq L$, and let $u^{\eps,\sigma}$ be the solutions to the variational
obstacle problems with the time  and space mollified obstacles $(\psi_\eps)_\sigma$.
Let $w^*$ denote the least solution with the obstacle $\psi$. Then
$$u^{\eps,\sigma} \to w^*, \quad \nabla u^{\eps,\sigma} \to \nabla w^*\quad \textrm{a.e.\ in}\quad
\Omega_T.$$
\end{theo}

\section{Special cases}

\label{sec:special-cases}
First, we consider the possibility to extend Definition 3.3 directly to the
irregular case. Needless to say, the variational inequality (1.1)
makes
 sense without the assumption
that the obstacle is continuous. However, the time derivative of the test function is present, and thus we might be led to use smooth or, at least, continuous test functions. We encounter a difficulty. It turns out that such a
restriction on the admissible test functions  destroys the
uniqueness property if the obs\-tacle is too irregular: there are too few test functions to detect the
``true solution''.

 To illustrate this, we consider the elliptic
obs\-tacle problem.  Let $\psi\in W^{1,p}(\Om)$  and recall
\[
\begin{split}
\mathcal K_{\psi}=\{\phi\in W^{1,p}(\Om)\,:\,\phi\geq \psi \trm{ a.e.\ in } \Om,\, \phi-\psi\in W^{1,p}_0(\Om)\}.
\end{split}
\]
Then $w\in \mathcal K_{\psi}$ is a solution to the elliptic obstacle problem if
\begin{equation}
\label{eq:elliptic}
\int_{\Om} \abs{\nabla w}^ {p-2} \nabla w \cdot \nabla
(\phi-w) \ud x \geq 0
\end{equation}
for every $\phi\in \mathcal K_{\psi}$.

Let us begin our discussion with the simplest relevant special case, the Dirichlet integral. Thus $p = 2$, the equation is linear
and stationary. Even here the so-called Lavrentiev Phenomenon,
described in \cite{kilpelainenl95}, enters and will destroy the uniqueness, if continuity is
imposed on the admissible functions. Fix a function $\psi \in W^{1,2}(\Omega)$ and consider the class
$$\mathcal K_{\psi}=\{\phi\in W^{1,2}(\Om)\,:\,\phi\geq \psi \trm{
  a.e.\ in }
 \Om,\, \phi-\psi\in W^{1,2}_0(\Om)\}$$
of admissible functions.
If $\psi$ itself is a superharmonic function,
say $\psi = u$,
it solves the obstacle problem: for all $\phi \in \mathcal K_{u}$
$$\int_{\Omega}|\nabla u|^2\,dx \leq \int_{\Omega}|\nabla \phi|^2\,dx
,$$
or equivalently
$$\int_{\Omega}\nabla u \cdot (\nabla \phi - \nabla u)\,dx \geq 0.$$
According to \cite{kilpelainenl95} there exists a superharmonic function $u \in
W^{1,2}(\Omega)$
such that
$$\int_{\Omega}|\nabla u|^2\,dx < \inf_{\phi} \int_{\Omega}|\nabla
\phi|^2\,dx,$$
where we restrict ourselves to \emph{continuous} functions $\phi$ in $
\mathcal K_{u}$. Notice that the inequality is strict. Thus the true
minimum cannot be reached via continuous admissible functions. This is an
instance of the Lavrentiev Phenomenon. From now on $u$ denotes this function.

There exists another superharmonic
function $w$\, ($w \geq u$ everywhere
 and $w \not = u$ in a subset of positive measure) such that
$$\int_{\Omega}|\nabla w|^2\,dx =  \inf_{\phi} \int_{\Omega}|\nabla
\phi|^2\,dx,$$
where the infimum is taken over all $\phi \in C(\Omega) \cap \mathcal
K_{u}$. Also a.e.\
\begin{equation}
 w = \widehat{\inf v},
\end{equation}
where the infimum is taken over all \emph{continuous} superharmonic
functions $v$ such that $v \geq u$ a.e. in $\Omega$.

Now
$$\int_{\Omega} \nabla u \cdot \nabla (\phi - u)\,dx \geq 0$$
for all $\phi \in \mathcal K_{u}$ and \emph{a fortiori} for all
  $\phi \in  C(\Omega) \cap \mathcal K_{u}$.  We also have
$$\int_{\Omega} \nabla w \cdot \nabla (\phi - w)\,dx \geq 0$$
for all $\phi \in \mathcal K_{w}$.
We claim that this also holds  for
all  $\phi \in  C(\Omega) \cap \mathcal K_{u}$, where the class of test functions is now defined using $u$. To see this, notice
that
\[
\begin{split}
\int_{\Omega}& \nabla w \cdot \nabla (\phi - w)\,dx\\
 &= \int_{\Omega} \nabla w \cdot \nabla (\max(\phi,w) - w)\,dx
+\int_{\Omega} \nabla w \cdot \nabla (\min(\phi,w) - w)\,dx\\
&\geq 0 +\int_{\{\phi < w\}} \nabla w \cdot \nabla (\phi - w)\,dx.
\end{split}
\]
The set $\{\phi < w\}$ is open, and in any case $\phi \geq u.$
Therefore one can conclude that $w$, in fact, is a harmonic function
in this open set. {\small To see this, fix a point in this set. In a sufficiently
  small ball centered at this point, we can replace $w$ by the
  harmonic function with the boundary values $w$ on the sphere (this
  is given by Poisson's integral) without touching $\phi$; the local
  Poisson modification lies above $u$. If we now perform the same
  construction on each of the continuous superharmonic functions, the
  infimum of which appears in (5.2), we notice that locally $w$ is
  the limit of harmonic functions.} Thus the last integral is zero. This proves the claim.

The consequence of this construction is that the variational inequality
$$ \int_{\Omega} \nabla v \cdot \nabla (\phi - v)\,dx \geq 0$$
has (at least) two solutions in the class $ \mathcal K_{u}$, if merely
\emph{continuous} functions $\phi$ in  $ \mathcal K_{u}$ are
admissible. The solutions exhibited are $u$ and $w$. However, if $\phi$ runs
through the whole class  $ \mathcal K_{u}$, then $u$ is the unique
solution.

The same phenomenon occurs for the problem
$$\int_{\Omega}|\nabla v|^{p-2}\nabla v \cdot \nabla(\phi - v)\,dx
\geq 0.$$
 Using an obstacle of the form
$u(x,t) = u(x)$ we get \textsf{ a counterexample to uniqueness  for the parabolic case}, if
the admissible functions are required to be continuous.

In the light of the previous calculation, testing with
smooth functions is insufficient to obtain uniqueness  even in the elliptic case. On the
other hand, \eqref{eq:smooth} does not make sense if the test functions
have poor regularity in the time direction. This is the difficulty.

Next we consider two special  cases: upper semicontinuous
obstacles, including characteristic functions of compact sets, and
lower semicontinuous obstacles.

First, we observe that with the characteristic function $\chi_K$  of a
compact set $K$ as an obstacle, $w^*$  is $p$-parabolic and,  in
particular, continuous   in
$\Om_T\setminus \ol K$ by \lemref{lem:global-int}.
\begin{lem}\label{lem:global-int-characteristic}
Let $K\subset \Om_T$ be a compact set, and let $w^*$ be the least solution with the obstacle  $\chi_K$. Then
$w^*$ is $p$-parabolic  in $\Om_T\setminus K$. Moreover, $w^* \in L^p(0,T;W^{1,p}(\Om))$.
\end{lem}

Let us now consider a \emph{lower} semicontinuous obstacle and
approximate it pointwise from \emph{below} by smooth
functions. Solving the corresponding
 obstacle problems we obtain the least solution as a limit,  cf.\ Corollary~\ref{cor:smooth-variational=least}.
 Needless to say, this is no surprise.
\begin{prop}
\label{prop:lsc}
Suppose that the obstacle $\psi$, $0\leq \psi\leq L$, is lower semicontinuous in $\ol \Om_T$ and let $\psi_i$ be an increasing sequence of smooth functions so that
\[
\begin{split}
\psi_i\to \psi
\end{split}
\]
pointwise.  Let $u_i$ be the variational solutions with the obstacles $\psi_i$, and let $w^*$ be the least solution with the obstacle $\psi$. Then
 \[
  u_i\to w^*, \quad \nabla u_i\to \nabla w^* \textrm{ a.e.\ in}\quad \Om_T.
    \]
\end{prop}
\begin{proof}
This is a simple consequence of a comparison principle because it implies $u_i\leq w^*$, and on the other hand, clearly for the limit $u$ it holds that $\psi\leq u$. Since by our convergence results $u$ is a supersolution, $w^*\leq u$.

To be more precise, since $\psi_i$ is smooth, it follows that
$u_i=\psi_i$ at the boundary of the open set $\{u_i>\psi_i\}$ and
$u_i$ is $p$-parabolic in the set $\{u_i>\psi_i\}$. Furthermore, $w^*\geq \hat
\psi_i=\psi_i$ and,  due to the comparison principle, $u_i\leq
w^*$ in the set $\{u_i>\psi_i\}$.

The convergence of $u_i$ to some limit $u$ follows from \theoref{thm:superparab-compactness}. Since the reasoning above was independent of
$i$, it follows that $u\leq w^*$ in the whole domain. On the other hand, $u_i$ is an
increasing and bounded sequence and, clearly, $u\geq
\psi$. Therefore, the limit $u$ is a supersolution above $\psi$. It follows
that $w^*=u$ almost everywhere.
\end{proof}

\textsf{Counterexample:} The situation is not symmetric. A similar statement is clearly false for an
approximation of an  upper semicontinuous obstacle $\psi$  by smooth
functions from above, when one uses the  variational solutions for the
corresponding obstacle problems. To see this, take
\[
\begin{split}
\psi(x,t)=\begin{cases}
1,&(x,t)\in \Om\times \{\frac{T}{2}\}\\
0,&\trm{otherwise},
\end{cases}
\end{split}
\]
as an obstacle. (Further, one can define $\psi$ as zero near the
lateral boundary, so that it has compact support. This has no
bearing.) This $\psi = 0$ a.e., so clearly the least solution is identically zero, but
an approximation of $\psi$ from above produces a supersolution $u$
that is not identically zero. Indeed, one has the minorant
\[
\begin{split}
v(x,t)=\begin{cases}
0,& t \leq  \frac{T}{2}\\
h(x,t),& t > \frac{T}{2},
\end{cases}
\end{split}
\]
where $h$ is the $p$-parabolic function in $\Omega \times
(\frac{T}{2},T)$ with initial values $1$ at $t = T/2$ and lateral boundary
values 0.

  Notice also that both $u$
and $\psi$ satisfy  \defiref{def:smooth} when testing with continuous
test functions \emph{everywhere} above the obstacle, so clearly
\textsf{uniqueness fails} with these test functions.
It is $u$ that is the variational  solution resulting from the approximation procedure, because it is plain that $\psi_{\eps} = 0$. Thus it is also the least
solution. For the non-uniqueness
it was essential to use continuous test functions  satisfying $\phi\geq \psi$ at
\emph{each} point, although $\psi$ is discontinuous.

The example also shows that \textsf{the convolutions} $\psi_{\eps}$ \textsf{cannot be
replaced} (in \theoref{thm:uniqueness}) \textsf{by arbitrary smooth obstacles}, say $\psi_{j}$ converging to
$\psi$ in the Sobolev space $L^p(0,T;W^{1,p}(\Omega))$.

\medskip
As we already have pointed out, the theory of thin obstacles
is outside the scope of our work, see \cite{petersson06}. However, we include the following
considerations.
 If we strengthen  \emph{almost everywhere} in the definition of a
 least solution to the requirement that the inequalities hold at \emph{each}
 point, then we can avoid the phenomenon in the
 counterexample. However, we must restrict ourselves to a
 \emph{semicontinuous obstacle} in this situation.

Thus we temporarily use
 the smaller class
\begin{equation}
\label{eq:everywhere-above}
\begin{split}
\mathcal S_{\psi}^{\#} =\{&u\,:\, u\textrm{ is $\essliminf$-regularized weak supersolution},\\
&\hspace{14 em}u\geq\psi \,\, \textrm{at each point}\}.
\end{split}
\end{equation}
to define the function $w^*_{\#}.$
Instead, we then obtain the following result.
\begin{prop}
\label{prop:usc}
Suppose that the obstacle $\psi$, $0\leq \psi\leq L$, is upper semicontinuous in $\ol \Om_T$ and define the least solution $w^*_{\#}$, using \eqref{eq:everywhere-above}. Further, let $\psi_i$ be a decreasing  sequence of smooth obstacles so that
\[
\begin{split}
\psi_i\to \psi
\end{split}
\]
pointwise. Then for the  variational solutions $u_i$ with the obstacles $\psi_i$,  it holds that
 \[
  u_i\to w^*_{\#}, \quad
  \nabla u_i\to \nabla w^*_{\#} \quad \textrm{a.e.\ in}\quad \Om_T.
    \]
\end{prop}
\begin{proof}
The idea in the proof is to extract, by the definition of the least solution, a decreasing  sequence of lower semicontinuous supersolutions converging to $w^*_{\#}$. By lower semicontinuity of these supersolutions and upper semicontinuity of the obstacle, there exists a continuous obstacle in between. This yields a sequence of continuous solutions, and upon a second approximation procedure by smooth obstacles, we can pass to a sequence of smooth solutions.

Next we work out the details. The proof of \theoref{thm:obstacle-sol-supersol} yields a sequence $v_i$, $v_i\geq \psi$,
 of $\essliminf$-regularized supersolutions converging almost everywhere to  $w^*_{\#}$. Since $\psi$ is upper semicontinuous and $v_i$ lower semicontinuous, there exists a continuous $\tilde \psi_i$ in $\ol \Om_T$ such that
\[
\begin{split}
\psi\leq\tilde \psi_i\leq v_i
\end{split}
\]
as shown in~\cite{hahn17}.
Denote the continuous least solutions with the obstacles $\tilde \psi_i$ by $\tilde u_i$. It follows that
\[
\begin{split}
\tilde u_i\to w^*_{\#}
\end{split}
\]
 almost everywhere because it immediately follows that $w^*_{\#} \leq \tilde u_i\leq v_i$. Further, \theoref{thm:superparab-compactness} implies the convergence of the gradients.

Remember that $\tilde u_i$ is continuous, and choose for every index $i$
a decreasing sequence $\psi^i_j$ of smooth obstacles such that
\[
\begin{split}
\psi^i_j\to \tilde u_i
\end{split}
\]
uniformly as $j\to \infty$.  Fix $\eps>0$ and choose a $\psi^i_{j}$
such that $\tilde u_i +\eps\geq\psi^i_{j}$. Thus $j =
j(i,\eps)$. Denote by $u^i_{j}$  the variational solution with the obstacle
$\psi^i_{j}$. Since $\tilde u_i+\eps\geq \psi^i_{j}$ and $\tilde
u_i+\eps$ is a continuous supersolution, it follows
by comparison that
\[
\begin{split}
\tilde u_i+\eps\geq u^i_{j}\geq \psi^i_j\geq \tilde u_i.
\end{split}
\]
 By a diagonalization argument, we can extract a subsequence of smooth obstacles so that the related
 solutions converge to some $u$ such that $w^*_{\#}+\eps\geq
 u\geq w^*_{\#}$ almost everywhere. By letting  $\eps\to 0$  via a
 subsequence $\eps_{k}$ and diagonalizing once more, we can extract a
 new subsequence $\psi'_k$ with corresponding solutions $u'_k$, converging to $w^*_{\#}$ in the sense of the claim.

To finish the proof, it is enough to notice that for any $\delta>0$ and $\psi'_k$, it holds for all $j$ large enough that  $\psi_j \leq \psi'_k+\delta$, where $\psi_j$ refers to the sequence in the statement of the proposition.
\end{proof}

\medskip
{\bf Acknowledgements.} A preliminary version of this work was
accomplished in May 2007 while M.P.\ visited the Norwegian University
of Science and Technology. The authors are
grateful to Tero Kilpeläinen for useful discussions and to Giuseppe
Mingione for his hospitality during 'Nonlinear Problems in PDEs'
-Intensive Research Period at Parma in 2010.

%\bibliography{citations2}

\begin{thebibliography}{BDGO97}

\bibitem[AL83]{altl83}
H.~W. Alt and S.~Luckhaus.
\newblock Quasilinear elliptic-parabolic differential equations.
\newblock {\em Math. Z.}, 183(3):311--341, 1983.

\bibitem[BDGO97]{boccardodgo97}
L.~Boccardo, A.~Dall'Aglio, T.~Gallou{\"e}t, and L.~Orsina.
\newblock Nonlinear parabolic equations with measure data.
\newblock {\em J. Funct. Anal.}, 147(1):237--258, 1997.

\bibitem[BDM]{bogeleindm09}
V.~B\"ogelein, F.~Duzaar, and G.~Mingione.
\newblock Degenerate problems with irregular obstacles.
\newblock To appear in J. Reine Angew. Math.

\bibitem[DiB93]{dibenedetto93}
E.~DiBenedetto.
\newblock {\em Degenerate parabolic equations}.
\newblock Universitext. Springer-Verlag, New York, 1993.

\bibitem[Hah17]{hahn17}
H.~Hahn.
\newblock {\"U}ber halbstetige und unstetige {F}unktionen.
\newblock {\em Sitzungsberichte Akad.\ Wiss.\ Wien Abt.\ II a}, 126:91--110,
  1917.

\bibitem[HKM93]{heinonenkm93}
J.~Heinonen, T.~Kilpel{\"a}inen, and O.~Martio.
\newblock {\em Nonlinear {P}otential {T}heory of {D}egenerate {E}lliptic {E}quations}.
\newblock Oxford Mathematical Monographs. Oxford University Press, New York,
  1993.

\bibitem[Kil89]{kilpelainen89}
T.~Kilpel{\"a}inen.
\newblock Potential theory for supersolutions of degenerate elliptic equations.
\newblock {\em Indiana Univ. Math. J.}, 38(2):253--275, 1989.

\bibitem[KKP10]{kortekp10}
R.~Korte, T.~Kuusi, and M.~Parviainen.
\newblock A connection between a general class of superparabolic functions and
  supersolutions.
\newblock {\em J. Evol. Equ.}, 10(1):1--20, 2010.

\bibitem[KKS09]{korteks09}
R.~Korte, T.~Kuusi, and J.~Siljander.
\newblock Obstacle problem for nonlinear parabolic equations.
\newblock {\em J. Differential Equations}, 246(9):3668--3680, 2009.

\bibitem[KL95]{kilpelainenl95}
T.~Kilpel{\"a}inen and P.~Lindqvist.
\newblock The {L}avrentiev phenomenon and the obstacle problem for the
  {D}irichlet integral.
\newblock {\em Proc. Amer. Math. Soc.}, 123(8):2459--2464, 1995.

\bibitem[KL96]{kilpelainenl96}
T.~Kilpel{\"a}inen and P.~Lindqvist.
\newblock On the {D}irichlet boundary value problem for a degenerate parabolic
  equation.
\newblock {\em SIAM J. Math. Anal.}, 27(3):661--683, 1996.

\bibitem[KL06]{kinnunenl06}
J.~Kinnunen and P.~Lindqvist.
\newblock Pointwise behaviour of semicontinuous supersolutions to a quasilinear
  parabolic equation.
\newblock {\em Ann. Mat. Pura Appl. (4)}, 185(3):411--435, 2006.

\bibitem[KLP10]{kinnunenlp10}
J.~Kinnunen, T.~Lukkari, and M.~Parviainen.
\newblock An existence result for superparabolic functions.
\newblock {\em J. Funct. Anal.}, 258:713--728, 2010.

\bibitem[KS]{kinderlehrers80}
D.~Kinderlehrer and G.~Stampacchia.
\newblock {\em An {I}ntroduction to {V}ariational {I}nequalities and {T}heir
  {A}pplications}, volume~88 of {\em Pure and Applied Mathematics}.
\newblock Academic Press Inc.

\bibitem[Kuu09]{kuusi09}
T.~Kuusi.
\newblock Lower semicontinuity of weak supersolutions to nonlinear parabolic
  equations.
\newblock {\em Differential Integral Equations}, 22(11-12):1211--1222, 2009.

\bibitem[LM94]{lim94}
G.~B. Li and O.~Martio.
\newblock Stability in obstacle problems.
\newblock {\em Math. Scand.}, 75(1):87--100, 1994.

\bibitem[LM07]{lindqvistm07}
P.~Lindqvist and J.~J. Manfredi.
\newblock Viscosity supersolutions of the evolutionary $p$-{L}aplace equation.
\newblock {\em Differential Integral Equations}, 20(11):1303--1319, 2007.

\bibitem[Nau84]{naumann84}
J.~Naumann.
\newblock {\em Einf\"uhrung in die {T}heorie parabolischer
  {V}ariationsungleichungen}, volume~64 of {\em Teubner-Texte zur Mathematik}.
\newblock BSB B. G. Teubner Verlagsgesellschaft, Leipzig, 1984.

\bibitem[Pet06]{petersson06}
C.~Petersson.
\newblock Continuity of parabolic {$Q$}-minima under the presence of irregular
  obstacles.
\newblock {\em Adv. Differential Equations}, 11(12):1397--1436, 2006.

\bibitem[Sim87]{simon87}
J.~Simon.
\newblock Compact sets in the space {$L\sp p(0,T;B)$}.
\newblock {\em Ann. Mat. Pura Appl. (4)}, 146:65--96, 1987.

\bibitem[WZYL01]{wuzyl01}
Z.~Wu, J.~Zhao, J.~Yin, and H.~Li.
\newblock {\em Nonlinear {D}iffusion {E}quations}.
\newblock World Scientific Publishing Co. Inc., River Edge, NJ, 2001.

\end{thebibliography}
%\bibliographystyle{alpha}

\def\cprime{$'$} \def\cprime{$'$}

\end{document}